\documentclass[12pt,a4paper,final]{amsart}
\usepackage{amssymb}
\usepackage{hyperref} 
\usepackage[notcite, notref]{showkeys}
\usepackage[utf8]{inputenc}

\theoremstyle{plain}
\newtheorem{theorem}{Theorem}[section]
\newtheorem{proposition}[theorem]{Proposition}
\newtheorem{corollary}[theorem]{Corollary}
\newtheorem{lemma}[theorem]{Lemma}

\theoremstyle{definition}

\newtheorem{example}[theorem]{Example}
\newtheorem{remark}[theorem]{Remark}

\theoremstyle{remark}

\numberwithin{equation}{section}

\newcommand{\N}{\mathbb N}
\newcommand{\Z}{\mathbb Z}

\newcommand{\R}{\mathbb R}

\DeclareMathOperator{\Ot}{O}

\newcommand{\Id}{\textup{Id}}

\DeclareMathOperator{\diag}{diag}

\DeclareMathOperator{\Aut}{Aut}

\title[Isospectral and not strongly isospectral spherical space forms]
{Isospectral and not strongly isospectral spherical space forms with non-cyclic fundamental group}

\author{Mauro Colantonio}
\address{Instituto de Matemática (INMABB), Departamento de Matemática, Universidad Nacional del Sur (UNS)-CONICET, Bahía Blanca, Argentina.}
\email{mauro.colantonio@uns.edu.ar}

\thanks{The author was supported by a grant from SGCYT--UNS (PGI 24/L126).}
\date{\today}

\begin{document}

\begin{abstract}
We construct the first examples of isospectral but not strongly isospectral spherical space forms with non-cyclic fundamental groups. 
We obtain infinitely many such pairs with isomorphic fundamental groups of Type I. 
	
\end{abstract} 
	
	\maketitle
	
	
\section{Introduction}

Every Riemannian manifold $(M,g)$ has an associated Laplace--Beltrami operator $\Delta_g$. When $M$ is closed (compact and without boundary), its spectrum $\operatorname{Spec}(M,g)$ is a discrete sequence of non-negative real numbers, each repeated finitely many times according to its multiplicity. Two compact Riemannian manifolds are said to be \emph{isospectral} if their Laplace--Beltrami spectra coincide. 
A stronger notion is obtained by simultaneously considering all natural elliptic differential operators:
two closed Riemannian manifolds are called \emph{strongly isospectral} if every natural elliptic differential operator acting on the same natural vector bundle over both manifolds have the same spectra. 
In particular, strong isospectrality implies $p$-isospectrality (i.e.\ their Hodge-Laplacians on $p$-forms have the same spectra) for every degree $p$, and hence ordinary isospectrality (i.e.\ $p=0$). The converse does not hold in general (see \cite{LMR-onenorm}).

In this article, we focus on isospectrality among spherical space forms, that is, Riemannian manifolds of the form $S^q/G$, where $G$ is a finite subgroup of $O(q+1)$ acting freely on the unit sphere $S^q$. The quotient is endowed with constant sectional curvature equal to one, and its fundamental group is isomorphic to $G$. 

The first examples of non-isometric isospectral spherical space forms were constructed by Ikeda in \cite{Ikeda80_isosp-lens}. They were lens spaces and therefore had cyclic fundamental groups. Many further examples of isospectral lens spaces have since been obtained; see \cite{LMR-SaoPaulo} and the references therein.

The situation for spherical space forms with non-cyclic fundamental groups has been different. Ikeda \cite{Ikeda83} constructed the first non-isometric isospectral spherical space forms with non-cyclic fundamental groups. Gilkey~\cite{Gilkey85} subsequently proved that Ikeda's examples are in fact strongly isospectral. Later, Wolf~\cite{Wolf01} developed a general construction of strongly isospectral spherical space forms having a fixed fundamental group. In particular, he shows that any two irreducible spherical space forms with isomorphic fundamental groups are strongly isospectral. Here, \emph{irreducible} means the spherical space form is of the form $S^q/\rho(\Gamma)$ with $\Gamma$ a finite group and $\rho$ an irreducible orthogonal representation of degree $q+1$ satisfying that $\rho(\Gamma)$ acts freely on $S^q$.

More recently, \cite{ColantonioLauret-heatTypeI} shows that the isomorphism class of the fundamental group is not determined by the Laplace spectrum. More precisely, there exist infinitely many pairs of strongly isospectral spherical space forms with non-isomorphic fundamental groups.
 
 Summing up, all previously known examples of isospectral spherical space forms with non-cyclic fundamental groups arise from \cite{Ikeda83}, \cite{Gilkey85}, \cite{Wolf01}, or \cite{ColantonioLauret-heatTypeI}, and all of them are strongly isospectral.
 
 A natural question (see \cite[Question~5.13]{LMR-survey}, and  \cite[Question~5.5]{ColantonioLauret-heatTypeI}) is whether there are isospectral spherical space forms with non-cyclic fundamental groups that are not strongly isospectral. 
 The next result respond it positively.
 
 \begin{theorem}\label{thm:introduction-main}
 		There exist infinitely many pairs of spherical space forms with non-cyclic isomorphic fundamental group that are isospectral but not strongly isospectral. 
 \end{theorem}

The article is organized as follows. 
Section~2 recalls the necessary background on spherical space forms, their spectra, lens spaces, and fixed point free groups of Type~I. It also presents explicit criteria, due to Wolf, for isometry (Theorem~\ref{thm:caracterizacionTypeIsphericalspaceforms}) and strong isospectrality (Theorem~\ref{thm:Wolf-strong-isospectrality}). In particular, the latter provides an explicit formulation of Wolf's strong isospectrality criterion for Type~I groups that is not stated in this form in \cite{Wolf01}.
In Section~3, we focus on spherical space forms with a particular non-cyclic fundamental group of Type~I and derive associated lens spaces and their spectral generating functions. 
We then establish a general criterion to produce isospectral spherical space forms that are not strongly isospectral and conclude with the explicit infinite family proving Theorem~\ref{thm:introduction-main}.

\subsection*{Acknowledgments}

The author would like to thank his Ph.D.\ advisor, Emilio Lauret, for his guidance and support throughout this work.

\section{Preliminaries}\label{sec:preliminaries}
	
In this section we recall the necessary background on spherical space forms with fundamental group of Type~I, the group representations that define them and their spectra. We refer the reader to \cite{Wolf-book} for a comprehensive treatment and to \cite{Wolf01} for a summarized version (see also \cite[\S2]{LMR-SaoPaulo}).

	\subsection{Spherical space forms and their spectra}\label{subsec:sphericalspaceforms}
	A \emph{spherical space form} is a complete Riemannian manifold of constant positive sectional curvature, which we normalize to be equal to~$1$. By the Killing--Hopf theorem, any such manifold is isometric to a quotient $S^q/G$, where $S^q$ is the unit sphere in $\R^{q+1}$ and $G$ is a finite subgroup of $\Ot(q+1)$ acting freely on $S^q$. The latter condition means that $1$ is not an eigenvalue of $g$ for any $g \in G \smallsetminus \{\Id\}$, or equivalently, that $g$ has no fixed points on $S^q$. The fundamental group of $S^q/G$ is isomorphic to $G$.
	Two spherical space forms $S^q/G_1$ and $S^q/G_2$ are \emph{isometric} if and only if $G_1$ is conjugate to $G_2$ in $\Ot(q+1)$.
	
	A finite group $\Gamma$ is called \emph{fixed point free} if it admits a \emph{fixed point free representation}, that is, a faithful representation $\rho\colon \Gamma \to \Ot(q+1)$ such that $\rho(\gamma)$ has no eigenvalue equal to~$1$ for any $\gamma \in \Gamma \setminus \{e\}$. In this case, $S^q/\rho(\Gamma)$ is a spherical space form with fundamental group isomorphic to~$\Gamma$;
	$S^q/\rho(\Gamma)$ is called \emph{irreducible} if $\rho$ is irreducible. 
	Two $q$-dimensional spherical space forms $S^q/\rho_1(\Gamma)$ and $S^q/\rho_2(\Gamma)$ are isometric if and only if there exists an automorphism $\psi \in \Aut(\Gamma)$ such that $\rho_1 \simeq \rho_2 \circ \psi$ (equivalence of representations).
	
	Wolf gave a complete classification of spherical space forms, organized in two stages. 
	The first determines the fixed point free finite groups.
	The second classifies, for each such group, its fixed point free real representations up to equivalence and the action induced by group automorphisms; see \cite[Chapters 5--7]{Wolf-book}.
	
	We next recall the spectrum of a spherical space form \(S^q/G\).  Let \(\mathcal{H}_k\) be the space
	of harmonic homogeneous polynomials of degree $k$ on $\R^{q+1}$. The group $G$ acts on $\mathcal{H}_k$ by
$
		(g\cdot P)(x)=P(g^{-1}x).
$
	We denote by $\mathcal{H}_k^{G}$ the subspace of $G$-invariant elements in $\mathcal H_k$. 
The spectrum of the Laplace-Beltrami operator on $S^q/G$ is given by the eigenvalues $k(k+q-1)$ with multiplicity $\dim\mathcal H_k^G$, for any non-negative integer $k$. 
See for instance \cite[\S2]{LMR-SaoPaulo} for more details.
Thus the spectrum is encoded by the generating function
	\begin{equation}
		F_G(z):=\sum_{k\geq 0}\dim(\mathcal{H}_k^G)\, z^k.
	\end{equation}
Sometimes, we will write $F_{S^q/G}(z)$ instead of $F_G(z)$, referring to the quotient manifold.

Using Molien's formula, one obtains (see \cite{Ikeda80_3-dimI} or \cite[Prop.~2.5]{Wolf01})
	\begin{equation}\label{eq:Ikeda-function}
		F_G(z)=\frac{1-z^2}{|G|} \sum_{g\in G}\frac{1}{\det(\Id-zg)}.
	\end{equation}
	Consequently,
two spherical space forms \(S^q/G_1\) and \(S^q/G_2\) are
isospectral (i.e.\ the spectra of their Laplace-Beltrami operators coincide) if and only if \(F_{G_1}(z)=F_{G_2}(z)\).

For any integer $p$ with $0\leq p\leq q$, two closed Riemannian manifolds are \emph{$p$-isospectral} if their Hodge-Laplace operators on $p$-forms have the same spectrum, including multiplicities.  In particular, $0$-isospectrality is ordinary isospectrality.
More generally, they are \emph{strongly isospectral} if the corresponding natural elliptic differential operators on every natural vector bundle have the same spectra.
Strong isospectrality therefore implies $p$-isospectrality for every $p$. 

The next result characterizes strongly isospectral spherical space forms. 
The converse direction follows using the generalized Sunada Theorem (see e.g.\ \cite[Prop.~2.10]{Wolf01}), and the forward direction was established by Pesce~\cite[Prop.~III.1]{Pesce95}.

\begin{theorem}\label{thm:stronglyisospectrality}
Given two spherical space forms $S^q/G_1$ and $S^q/G_2$, they are strongly isospectral if and only if $G_1$ and $G_2$ are almost conjugate in $\Ot(q+1)$. 
\end{theorem}

	\subsection{Type~I groups}\label{subsec:TypeIgroups}
	
	Following \cite{Wolf-book}, a finite group is of \emph{Type~I} if all its Sylow subgroups are cyclic. By a theorem of Burnside (see \cite[Thm.~5.4.1]{Wolf-book}), every fixed point free 
	Type~I group is described as follows.
	
	\begin{lemma}[{\cite[Thm.~5.4.1]{Wolf-book}}]\label{lem:typeI}
		Let $m$, $n$, $r$ be positive integers satisfying $\gcd((r-1)n,m)=1$ and $r^n \equiv 1 \pmod{m}$. Then the group $\Gamma$ generated by elements $A$, $B$ with relations
		\begin{equation}\label{eq:relations}
			A^m=B^n=1 \quad \text{and} \quad BAB^{-1}=A^r
		\end{equation}
		has order $mn$ and is of Type~I. Conversely, every Type~I group is isomorphic to one of this form.
	\end{lemma}
	
We denote by $\Gamma_d(m,n,r)$ the Type~I group described in Lemma~\ref{lem:typeI}, where $d$ denotes the order of $r$ in $\mathbb{Z}_m^\times$. (In this paper, we abbreviate $\Z_m=\Z/m\Z$ as ring and $\Z_m^\times$ means the group of invertible elements in $\Z_m$.) 
Every element of $\Gamma_d(m,n,r)$ can be written uniquely as $A^aB^b$, with $0 \leq a < m$ and $0 \leq b < n$; in particular,
$
|\Gamma_d(m,n,r)|=mn.
$
Moreover, $m$ is necessarily odd, since $\gcd(m,r-1)=1$, and $\Gamma_d(m,n,r)$ is cyclic if and only if $d=1$.

\begin{remark}\label{rem:fixedpointfreeTypeIgroups}
Such a group is fixed point free if and only if every prime divisor of $d$ also divides $n/d$ (see \cite[Thm.~5.4.1]{Wolf-book})
\end{remark}

	The irreducible complex fixed point free representations of a Type~I group 
	$\Gamma_d(m,n,r)$ are described as follows (see \cite[Thm.~5.5.6]{Wolf-book}).
	
	\begin{theorem} 
		\label{thm:irreps}
		Let $\Gamma_d(m,n,r)$ be a Type~I group with generators $A$, $B$ satisfying 
		\eqref{eq:relations}. For integers $k$, $\ell$ with $\gcd(k,m)=1=\gcd(\ell,n)$, let $\pi_{k,\ell}$ be the complex representation of degree $d$ defined by
		\begin{equation}\label{eq:complex_irrep}
			\pi_{k,\ell}(A) = \diag\!\left(e^{2\pi ik/m},\, e^{2\pi ikr/m},\, \ldots,\, 
			e^{2\pi i kr^{d-1}/m}\right), \quad
			\pi_{k,\ell}(B) = \begin{pmatrix} 0 & \Id_{d-1} \\ e^{2\pi i \ell d/n} & 0 \end{pmatrix}.
		\end{equation}
Then $\pi_{k,\ell}$ is an irreducible fixed point free complex representation, and $\pi_{k,\ell} \simeq \pi_{k',\ell'}$ if and only if $\ell \equiv \ell' \pmod{n/d}$ and $k \equiv k' r^c \pmod{m}$ for some integer $c$. 
Moreover, every irreducible fixed point free complex representation of $\Gamma_d(m,n,r)$ is equivalent to $\pi_{k,\ell}$ for some $k,\ell$. 
\end{theorem}

As a direct consequence of Theorem~\ref{thm:irreps}, every real fixed point free representation $\rho$ of
$\Gamma_d(m,n,r)$ is equivalent to a direct sum of the form
\begin{equation}
	\rho \simeq
	\bigoplus_{j=1}^{p}
	\left(
	\pi_{k_j,\ell_j}
	\oplus
	\overline{\pi}_{k_j,\ell_j}
	\right),
\end{equation}
for some integers $k_1,\ell_1,\dots,k_p,\ell_p$ satisfying $\gcd(k_j,m)=1=\gcd(\ell_j,n)$ for every $j$.
Each irreducible real summand $\pi_{k_j,\ell_j} \oplus \overline{\pi}_{k_j,\ell_j}$ has degree $2d$, and hence $\rho$ has
degree $2dp$.

Combining Theorem \ref{thm:irreps} and the preceding description with
the characterization of isometric spherical space forms with a given
fundamental group recalled in Subsection~\ref{subsec:sphericalspaceforms}, we obtain the following statement, which appears in the first paragraph of \cite[\S5.6]{Wolf-book}.

\begin{theorem}[Characterization of isometric spherical space forms of Type I]
\label{thm:caracterizacionTypeIsphericalspaceforms}
		Let $\Gamma=\Gamma_d(m,n,r)$ be a Type I group, $p$ a positive integer, and
		\begin{equation*}
			\rho= \bigoplus_{j=1}^{p} \left(\pi_{k_j,\ell_j}\oplus \overline{\pi}_{k_j,\ell_j} \right)
\qquad\text{and}\qquad
			\rho'= \bigoplus_{j=1}^{p} \left(\pi_{k'_j,\ell'_j}\oplus \overline{\pi}_{k'_j,\ell'_j} \right)
		\end{equation*}
two fixed point free real representations of $\Gamma$. Then the spherical space forms $S^{2dp-1}/\rho(\Gamma)$ and $S^{2dp-1}/\rho'(\Gamma)$ are isometric if and only if there exist integers $a$ and $b$ with 
\begin{align}\label{eq:conditions-a-b}
\gcd(a,m)&=1, &
\gcd(b,n)&=1, &
b&\equiv 1\pmod d,
\end{align}
a permutation $\sigma$ of $\{1,\ldots,p\}$,  $\epsilon_1,\dots,\epsilon_p\in\{-1,1\}$, and $c_1,\dots,c_p\in\{0,\ldots,d-1\}$ such that 
		\begin{align*}
			k_{\sigma(j)}' &\equiv \epsilon_j ak_{j}r^{c_j} \pmod m,&
			\ell_{\sigma(j)}' &\equiv \epsilon_j b\ell_{j} \pmod {n/d},
			\qquad 
			\text{for all $1\leq j\leq p$}.
		\end{align*}
		
\end{theorem}

	\subsection{Isospectral spherical space forms of Type I}\label{subsec:wolf_typeI}
	
	In \cite{Wolf01}, Wolf obtained sufficient conditions to produce strongly isospectral spherical space forms with isomorphic fundamental groups of any type. 
	The next statement, focused in spherical space forms of Type I, follows immediately from Proposition 6.3 and Theorem 6.4 in \cite{Wolf01}.

\begin{theorem}[Wolf] 
\label{thm:Wolf-strong-isospectrality}
Let $\Gamma=\Gamma_d(m,n,r)$ be a Type I group, $p$ a positive integer, and
\begin{equation*}
	\rho= \bigoplus_{j=1}^{p} \left(\pi_{k_j,\ell_j}\oplus \overline{\pi}_{k_j,\ell_j} \right)
	\qquad\text{and}\qquad
	\rho'= \bigoplus_{j=1}^{p} \left(\pi_{k'_j,\ell'_j}\oplus \overline{\pi}_{k'_j,\ell'_j} \right)
\end{equation*}
two fixed point free real representations of $\Gamma$. 
Suppose that there exist integers $a$ and $b$ with $\gcd(a,m)=1$, $\gcd(b,n)=1$, a permutation $\sigma$ of $\{1,\ldots,p\}$,  $\epsilon_1,\dots,\epsilon_p\in\{-1,1\}$, and integers $c_1,\dots,c_p\in\{0,\ldots,d-1\}$ such that 
\begin{align*}
	k_{\sigma(j)}' &\equiv \epsilon_j ak_{j}r^{c_j} \pmod m,&
	\ell_{\sigma(j)}' &\equiv \epsilon_j b\ell_{j} \pmod {n/d},
	\qquad 
	\text{for all $1\leq j\leq p$}.
\end{align*}
Then, the spherical space forms
$S^{2dp-1}/\rho(\Gamma)$ and $S^{2dp-1}/\rho'(\Gamma)$ are strongly isospectral.
\end{theorem}

\begin{remark}\label{rem:equal-condition}
Note that the condition $b\equiv 1\pmod d$ in \eqref{eq:conditions-a-b} is omitted in Theorem~\ref{thm:Wolf-strong-isospectrality}. This is essentially the only difference between the hypotheses of Theorem~\ref{thm:caracterizacionTypeIsphericalspaceforms} and Theorem~\ref{thm:Wolf-strong-isospectrality}.
\end{remark}

The results of Wolf for the other types are very similar, but more technical. However, when one assumes $p=1$ (i.e., $\rho$ and $\rho'$ are irreducible), the hypotheses are trivially satisfied in all types, yielding the following statement (see \cite[Cor.~7.2]{Wolf01}).

\begin{theorem}\label{thm:Wolf-strong-isospectrality-irreducible}
	Let $\Gamma$ be any finite fixed point free group, and let $\rho$ and $\rho'$ be irreducible real fixed point free representations of $\Gamma$. Then the spherical space forms $S^q/\rho(\Gamma)$ and $S^q/\rho'(\Gamma)$ are strongly isospectral.
\end{theorem}

The corresponding result for ordinary isospectrality and $\Gamma$ of Type I had previously been established by Ikeda~\cite[Thm.~1]{Ikeda83}.

\subsection{Lens spaces}\label{subsec:lensspaces}
For the purposes of this paper, we will refer to any spherical space form with cyclic fundamental group as a \emph{lens space}.
For $m,s_1,\dots,s_h$ integers with $m>0$ and $\gcd(m,s_j)=1$ for all $j$, we denote by $L(m; s_1,\ldots,s_h)$ the spherical space form $S^{2h-1}/G$ with $G$ the cyclic group generated by the matrix
\begin{equation}
	\left(
	\begin{array}{ccc}
		R(s_1/m) &        & 0\\
		& \ddots &  \\
		0        &        & R(s_h/m)
	\end{array}
	\right).
\end{equation}
where 
$ 
	R(\theta)=
	\left(\begin{smallmatrix}
		\cos 2\pi\theta & \sin 2\pi\theta\\
		-\sin 2\pi\theta & \cos 2\pi\theta
	\end{smallmatrix}\right)
$ 
for any $\theta\in\R$.

Two lens spaces $L(m; s_1,\ldots,s_h)$ and $L(m; s_1',\ldots,s_h')$ are 
\emph{isometric} if and only if there exist $t\in \Z$ prime to $m$, $\epsilon_1,\dots\epsilon_h \in \{-1,1\}$ and a permutation $\sigma$ of $\{1,\ldots,h\}$ such that 
$$
s_{\sigma(j)}' \equiv \epsilon_j t s_{j} \pmod{m}
\qquad\text{for all }\, 1\leq j\leq h .
$$

Write $\zeta_m =e^{2\pi i/m}$. 
It follows from \eqref{eq:Ikeda-function} that 
\begin{equation}\label{eq:Ikedalens}
	F_L(z)=\frac{1-z^2}{m} \sum_{a=0}^{m-1} \prod_{j=1}^{h} \frac{1}{\left(1-z\zeta_m^{as_j}\right) \left(1-z\zeta_m^{-as_j}\right)}.
\end{equation}

There are several examples of non-isometric isospectral lens spaces 
(see e.g.\ \cite{Ikeda80_isosp-lens}, \cite{Ikeda89}, \cite{GornetMcGowan20}, \cite{LMR-onenorm}, \cite{DeFordDoyle18}, \cite{Lauret-computationalstudy}). 
However, strongly isospectral lens spaces are necessarily isometric (see \cite[Prop.~7.2]{LMR-onenorm}).
Indeed, by Theorem~\ref{thm:stronglyisospectrality}, strong isospectrality implies that the corresponding cyclic subgroups are almost conjugate. However, almost conjugate cyclic subgroups of $\Ot(2h-1)$ are necessarily conjugate, and hence the lens spaces are isometric.

\section{Main Theorem}

In this section, we introduce some particular Type~I groups and then prove Theorem~\ref{thm:introduction-main} by showing that certain reducible fixed point free representations give rise to spherical space forms with the same spectral generating function, although their images are not almost conjugate.

\subsection{A special class of Type I groups}

In the rest of the article we focus on Type~I groups of the form $\Gamma_2(m,4,r)$ for some $m,r\in\N$.
In the notation of Subsection~\ref{subsec:TypeIgroups}, we are picking $d=2$ and $n=4$. 
Since $r$ has order $d=2$ in $\Z_m^\times$, $r^2\equiv 1\pmod m$, so $m\mid (r-1)(r+1)$. 
However, $\gcd(m,r-1)=1$ by Lemma~\ref{lem:typeI}, thus $m\mid r+1$. 
We conclude that $r\equiv -1\equiv m-1 \pmod m$, thus $\Gamma_2(m,4,r)=\Gamma_2(m,4,m-1)$, which allows us to assume $r=m-1$ without loosing generality. 
Furthermore, it is easy to see that $\Gamma_2(m,4,m-1)$ is fixed point free from  Remark~\ref{rem:fixedpointfreeTypeIgroups}.

Summing up, we will assume for the rest of the paper that 
\begin{equation}\label{eq:Gamma_2(m,4,m-1)}
\Gamma=\Gamma_2(m,4,m-1),
\end{equation}
where $m$ is any positive odd integer number greater than $1$. 
The group $\Gamma$ is fixed point free and non-cyclic, has order $4m$, and is generated by elements $A$ and $B$ satisfying $A^m=B^4=1$ and $BAB^{-1}=A^{-1}$.

\begin{remark}\label{rem:orders}
One can check that $AB^2$ has order $2m$ in $\Gamma$, and the order of any element in $\Gamma\smallsetminus\langle AB^2\rangle $ divides $4$. 
This also follows from Claim~2 and the last paragraph of the proof of \cite[Thm.~3.3]{ColantonioLauret-heatTypeI}. 
\end{remark}

Let $\pi_{k,\ell}$ be an fixed point free irreducible complex representation of $\Gamma$ as in Theorem~\ref{thm:irreps}. 
We can assume $0< k< m$ and $0< \ell< n$ without loosing generality. 
Moreover, because $n/d=2$, we can assume $\ell=1$ up to equivalence of $\pi_{k,\ell}$ by Theorem~\ref{thm:irreps}.

We recall the notation $\zeta_m=e^{2\pi i/m}$. 
A direct computation using \eqref{eq:complex_irrep} gives
\begin{align}\label{eq:block_A}
\pi_{k,1}(A^a) &
	= \begin{pmatrix} \zeta_m^{ak} & 0 \\ 0 & \zeta_m^{-ak} \end{pmatrix},
&
\pi_{k,1}(A^a B) &
	= \begin{pmatrix} 0 & \zeta_m^{ak} \\ -\zeta_m^{-ak} & 0 \end{pmatrix},
&
\pi_{k,1}(B^2)&
	=-\Id_2.
\end{align}
The last identity is consistent with the facts $B^4=1$ and $B^2\neq 1$.

Any real irreducible fixed point free representation of $\Gamma$ is equivalent to one of the form 
\begin{equation}
	\rho_{k} := \pi_{k,1} \oplus \overline{\pi}_{k,1}
\end{equation}
for some $1\leq k\leq m-1$ with $\gcd(m,k)=1$, which has real degree $2d=4$. 
An arbitrary real fixed point free representation of $\Gamma$ is equivalent to one of the form 
$\rho_{k_1} \oplus \cdots \oplus \rho_{k_p}$,
for some integers $k_1,\dots,k_p$ satisfying $1\leq k_j\leq m-1$ and $\gcd(k_j,m)=1$ for all $j$.

We next observe that strongly isospectral spherical space forms with fundamental group isomorphic to $\Gamma$ produced by Theorem~\ref{thm:Wolf-strong-isospectrality} are in fact isometric. 

\begin{remark}
Suppose there are two real fixed point free representations of $\Gamma=\Gamma_2(m,4,m-1)$ satisfying the hypotheses in Theorem~\ref{thm:Wolf-strong-isospectrality}. 
In particular, there are integers $a,b$ satisfying certain conditions. 
Obviously, $b\equiv 1\pmod {d}$ (because $d=2$), hence the hypotheses of Theorem~\ref{thm:caracterizacionTypeIsphericalspaceforms} are also satisfied (see Remark~\ref{rem:equal-condition}), 
which ensures the isometry between the corresponding spherical space forms.
\end{remark}

We now derive an explicit expression of $F_{\rho(\Gamma)}(z)$ for a real representation $\rho=\rho_{k_1}\oplus\dots\oplus\rho_{k_p}$ of $\Gamma$ as above. 
In what follows, we set
\begin{equation}\label{eq:lens-space-m}
L_\rho = L(m;k_1,k_1,\ldots,k_p,k_p),
\end{equation}
which is a $(4p-1)$-dimensional lens space.

\begin{proposition}\label{prop:type-I-generating-function}
Let $\Gamma=\Gamma_2(m,4,m-1)$ and  $\rho=\rho_{k_1}\oplus\cdots\oplus\rho_{k_p}$, for some integers $m,k_1,\dots,k_p$ satisfying $m>0$, $m\equiv 1\pmod2$, and $\gcd(m,k_j)=1$ for all $j$. 
Then
\begin{equation}\label{eq:Ikeda-Gamma_2(m,4,m-1)}
F_{\rho(\Gamma)}(z)
= \frac{1}{4} \Big(F_{L_\rho}(z)+F_{L_\rho}(-z)\Big) + \frac{1-z^2}{2(1+z^2)^{2p}}.
	\end{equation}
\end{proposition}

\begin{proof}
The cyclic subgroup
$ 
	\langle AB^2\rangle = \{A^a:0\leq a<m\} \cup \{A^aB^2:0\leq a<m\}
$ 
of $\Gamma$ has order $2m$. 
Note that 
\begin{equation*}
\pi_{k_j}(AB^2)= 
-\pi_{k_j}(A)
= \zeta_{2m}^m \begin{pmatrix} \zeta_m^{k_j} & 0 \\ 0 & \zeta_m^{-k_j} \end{pmatrix}
= \begin{pmatrix} \zeta_{2m}^{2k_j+m} & 0 \\ 0 & \zeta_{2m}^{-2k_j+m} \end{pmatrix}
.
\end{equation*}
By \eqref{eq:block_A}, the $j$-th $(4\times4)$-block of $\rho(AB^2)$ is given by  
\begin{equation*}
	(\pi_{k_j,1}\oplus\overline{\pi}_{k_j,1})(AB^2)
= 
\begin{pmatrix} 
	R\big(\frac{2k_j+m}{2m}\big) \\ & R\big(\frac{-2k_j+m}{2m}\big)
\end{pmatrix}
.
\end{equation*}
We conclude that $\rho(\langle AB^2\rangle) = \langle -\rho(A)\rangle$ is precisely the cyclic group of order $2m$ defining the lens space 
\begin{equation}\label{eq:lens-space-2m}
\widetilde L_\rho:= L\left(2m;m+2k_1,m-2k_1,\ldots,m+2k_p,m-2k_p \right).
\end{equation}

The elements in $\Gamma\smallsetminus \langle AB^2\rangle$ are  $A^aB$ and $A^aB^3$ for $0\leq a<m$ (see the last paragraph of \cite[Thm.~3.3]{ColantonioLauret-heatTypeI}).
By \eqref{eq:block_A}, 
\begin{align*}
\det\left(\operatorname{Id}_{4p}-z\rho(A^aB)\right) &
= \prod_{j=1}^p \det\left(\operatorname{Id}_{2}-z\pi_{k_j}(A^aB)\right)
\det\left(\operatorname{Id}_{2}-z \overline{\pi_{k_j}(A^aB)}\right)
= (1+z^2)^{2p}.
\end{align*}
Similarly,  $\det\left(\operatorname{Id}_{4p}-z\rho(A^aB^3)\right)= \det\left(\operatorname{Id}_{4p}+z\rho(A^aB)\right) = (1+z^2)^{2p}$. 
Note that both terms are independent of $a$.

It follows from
\eqref{eq:Ikeda-function} that 
	\begin{equation}
		\begin{aligned}
F_{\rho(\Gamma)}(z)&
= \frac{1-z^2}{4m}
	\sum_{a=0}^{m-1}\sum_{b=0}^3
	\frac{1}{\det\big(\operatorname{Id}-z\, \rho(A^aB^b) \big)}
\\& 
= \frac{1-z^2}{4m}
	\left(
	\sum_{\gamma\in \langle AB^2\rangle}
	\frac{1}{\det\big(\operatorname{Id}-z\rho(\gamma) \big)}
	+
	\frac{2m}{(1+z^2)^{2p}}
	\right)
\\ & 
=
\frac{1}{2} F_{\widetilde L_\rho}(z)
+
\frac{1-z^2}{2(1+z^2)^{2p}}.
\end{aligned}
\end{equation}

It remains to show that $F_{\widetilde L_\rho}(z)=\frac12 (F_{L_\rho}(z)+F_{L_\rho}(-z))$. 
We have from \eqref{eq:Ikedalens} that 
\begin{align*}
F_{\widetilde L_\rho}(z) &
= \frac{1-z^2}{2m} 
\sum_{a=0}^{2m-1} \prod_{j=1}^{p} 
 \frac{1}{
 	\big(1-z\zeta_{2m}^{a(m+2k_j)}\big) 
 	\big(1-z\zeta_{2m}^{-a(m+2k_j)}\big)
 	\big(1-z\zeta_{2m}^{a(m-2k_j)}\big) 
 	\big(1-z\zeta_{2m}^{-a(m-2k_j)}\big)
 }
\\ & 
= \frac{1-z^2}{2m} 
\sum_{a'=0}^{m-1} \prod_{j=1}^{p} 
\frac{1}{
		\big(1-z\zeta_{m}^{2a'k_j}\big)^2 
		\big(1-z\zeta_{m}^{-2a'k_j}\big)^2
	}
	\qquad\text{(even terms $a=2a'$)}
\\ & \quad 
+ \frac{1-z^2}{2m} 
\sum_{a'=0}^{m-1} \prod_{j=1}^{p} 
\frac{1}{
		\big(1+z\zeta_{m}^{(2a'+1)k_j}\big)^2 
		\big(1+z\zeta_{m}^{-(2a'+1)k_j}\big)^2
	}
	\qquad\text{(odd terms $a=2a'+1$)}
\\ & 
= \frac{1-z^2}{2m} 
\sum_{a=0}^{m-1} \prod_{j=1}^{p} 
\frac{1}{\big(1-z\zeta_{m}^{ak_j}\big)^2 \big(1-z\zeta_{m}^{-ak_j}\big)^2}
\\ &\quad
+ \frac{1-z^2}{2m} 
\sum_{a=0}^{m-1} \prod_{j=1}^{p} 
\frac{1}{\big(1+z\zeta_{m}^{ak_j}\big)^2 \big(1+z\zeta_{m}^{-ak_j}\big)^2}
.
\end{align*}
The last identity follows because the maps $a'\mapsto 2a'$ and $a'\mapsto 2a'+1$ are bijections of $\Z_m$.
The proof is complete, since the last two terms are precisely $\frac12 F_{L_\rho}(z)$ and $\frac12F_{L_\rho}(-z)$, respectively.
\end{proof}

\begin{corollary}\label{cor:lens-reduction-m}
Let $\Gamma=\Gamma_2(m,4,m-1)$, $\rho=\rho_{k_1}\oplus\cdots\oplus\rho_{k_p}$, and $\rho'=\rho_{k_1'}\oplus\cdots\oplus\rho_{k_p'}$, for some integers $m,k_1,k_1',\dots,k_p,k_p'$ satisfying $m>0$, $m\equiv 1\pmod2$, and $\gcd(m,k_j)=1=\gcd(m,k_j')$ for all $j$. 
If $L_\rho$ and $L_{\rho'}$ are isospectral, then $S^{4p-1}/\rho(\Gamma)$ and	$S^{4p-1}/\rho'(\Gamma)$ are isospectral.
\end{corollary}

\begin{proof}
By assumption, $F_{L_\rho}(z)=F_{L_\rho'}(z)$, and hence
$F_{L_\rho}(-z)=F_{L_\rho'}(-z)$.
Now, \eqref{eq:Ikeda-Gamma_2(m,4,m-1)} yields
$F_{\rho(\Gamma)}(z)=F_{\rho'(\Gamma)}(z)$, which is the desired conclusion.
\end{proof}

\subsection{Isospectral and not strongly isospectral spherical space forms}

We are now in position to prove Theorem~\ref{thm:introduction-main}, as a consequence of the following result.

\begin{theorem}\label{thm:main}
Let $\Gamma=\Gamma_2(m,4,m-1)$, for some prime integer $m\geq13$, and set $p=\frac{m-5}{2}$. 
Let $T=\{\alpha,\beta\}$ and
$T'=\{\alpha',\beta'\}$ be two-element subsets of $U_m:= \{1,2,\ldots,\frac{m-1}{2}\}$ satisfying 
$2\alpha\not\equiv\pm\beta\pmod m$, 
$\alpha\not\equiv\pm2\beta\pmod m$,
$2\alpha'\not\equiv\pm\beta'\pmod m$, and
$\alpha'\not\equiv\pm2\beta'\pmod m$. 
Write 
\begin{align}
U_m\smallsetminus T &=\{k_1,\dots,k_p\},&
U_m\smallsetminus T' &=\{k_1',\dots,k_p'\},
\end{align}
and set $\rho=\rho_{k_1}\oplus \dots\oplus \rho_{k_p}$ and $\rho'=\rho_{k_1'}\oplus \dots\oplus \rho_{k_p'}$.

The spherical space forms $S^{2m-11}/\rho(\Gamma)$ and $S^{2m-11}/\rho'(\Gamma)$ are isospectral. 
Moreover, they are not strongly isospectral (and consequently non-isometric) if the lens spaces $L_\rho$ and $L_{\rho'}$ (see \eqref{eq:lens-space-m}) are not isometric. 
\end{theorem}

\begin{proof}
We first prove isospectrality between $S^{2m-11}/\rho(\Gamma)$ and $S^{2m-11}/\rho'(\Gamma)$. 
By Corollary~\ref{cor:lens-reduction-m}, it is enough to prove that $L_\rho$ and $L_{\rho'}$ are isospectral. 
By \eqref{eq:Ikedalens}, it clearly suffices to show that
	\begin{equation}\label{eq:main-lens-sums}
		\sum_{a=1}^{m-1}
		\frac{1}
		{\displaystyle\prod_{j=1}^p
			(1-z\zeta_m^{ak_j})^2(1-z\zeta_m^{-ak_j})^2}
		=
		\sum_{a=1}^{m-1}
		\frac{1}
		{\displaystyle\prod_{j=1}^p
			(1-z\zeta_m^{ak_j})^2(1-z\zeta_m^{-ak_j})^2}.
	\end{equation}
	
	Since $m$ is prime, multiplication by any
	$a\in\{1,\ldots,m-1\}$ permutes the nonzero residue classes modulo $m$.
	Consequently,
\begin{equation*}
		\prod_{j\in U_m}
		(1-z\zeta_m^{aj})(1-z\zeta_m^{-aj})
		=
		\Phi_m(z),
\end{equation*}
the $m$-th cyclotomic polynomial.  
Multiplying by $\Phi_m(z)$ to both sides,  \eqref{eq:main-lens-sums} is equivalent to 
\begin{equation}\label{eq:main-lens-sums2}
\sum_{a=1}^{m-1} 
\prod_{j\in\{\alpha,\beta\}}
	(1-z\zeta_m^{aj})^2(1-z\zeta_m^{-aj})^2
	= 
\sum_{a=1}^{m-1}
\prod_{j\in\{\alpha',\beta'\}}
(1-z\zeta_m^{aj})^2(1-z\zeta_m^{-aj})^2
.
\end{equation}

We now show that the polynomial $P(z)$ at the left-hand side of \eqref{eq:main-lens-sums2} depends only on $m$. 
The coefficient of $z^q$ is a sum, over $a=1, \ldots,m-1$, of terms of the form 	$(-1)^q\zeta_m^{a\lambda}$, where $\lambda$ is the sum of $q$ elements chosen from the multiset
\begin{equation*}
	W=[\alpha,\alpha,-\alpha,-\alpha,
	\beta,\beta,-\beta,-\beta].
\end{equation*}

We first determine which such sums can vanish modulo $m$. 
Any subsum of $W$ has the form
\[
u\alpha+v\beta,
\qquad u,v\in\{-2,-1,0,1,2\}.
\]
By using the hypotheses on $\alpha,\beta$, one can check that 
\begin{equation}\label{eq:no-nontrivial-zero-sums}
	u\alpha+v\beta\equiv0\pmod m
	\quad\Longrightarrow\quad
	u=v=0.
\end{equation}

Let $V_q$ denote the number of $q$-element selections from $W$, counted
with multiplicities, whose sum is congruent to zero modulo $m$. 
By \eqref{eq:no-nontrivial-zero-sums}, such a selection must contain the
same number of copies of $\alpha$ and $-\alpha$, and the same number of copies of $\beta$ and $-\beta$. 
A direct count gives
\begin{equation*}
	V_0=1,\qquad
	V_1=0,\qquad
	V_2=8,\qquad
	V_3=0,\qquad
	V_4=18,
\end{equation*}
and, by taking complementary selections,
$
V_q=V_{8-q}.
$

For $\lambda\in\mathbb{Z}$, the standard root-of-unity identity gives
\begin{equation}\label{eq:root-of-unity-sum}
	\sum_{a=1}^{m-1}\zeta_m^{a\lambda}
	=
	\begin{cases}
		m-1,&\lambda\equiv0\pmod m,\\
		-1,&\lambda\not\equiv0\pmod m.
	\end{cases}
\end{equation}
Consequently, the coefficient of $z^q$ in $P(z)$ is 
\[
(-1)^q\left((m-1)V_q-(\tbinom{8}{q}-V_q)\right)
=
(-1)^q\left(mV_q-\tbinom{8}{q}\right).
\]
Therefore,
\begin{equation}\label{eq:universal-polynomial}
\begin{aligned}
P(z)=
	&(m-1)+8z+(8m-28)z^2+56z^3+(18m-70)z^4
\\ &
+56z^5+(8m-28)z^6 +8z^7+(m-1)z^8.
\end{aligned}
\end{equation}
In particular, this polynomial depends only on $m$ and not on the
choice of $T=\{\alpha,\beta\}$.

The same argument applies to $T'=\{\alpha',\beta'\}$ by the analogous assumptions on $\alpha'$ and $\beta'$.
We conclude that \eqref{eq:main-lens-sums2} holds. Consequently, the associated lens spaces are isospectral, and Corollary~\ref{cor:lens-reduction-m}
shows that the two spherical space forms are isospectral.

Finally, suppose that the two spherical space forms are strongly isospectral. 
By Theorem~2.1, the subgroups $\rho(\Gamma)$ and $\rho'(\Gamma)$ are almost conjugate in $\Ot(2m-10)$. 
Since conjugate elements have the same order, almost conjugacy also holds after restricting to the elements of odd order, which are $\langle A\rangle$ by Remark~\ref{rem:orders}. 
It follows that the cyclic subgroups $\langle\rho(A)\rangle$ and $\langle\rho'(A)\rangle$ are almost conjugate in $\Ot(2m-10)$, as well as $\langle\rho(AB^2)\rangle= \langle-\rho(A)\rangle$ and $\langle\rho'(AB^2)\rangle= \langle-\rho'(A)\rangle$. 
Almost conjugate cyclic subgroups are necessarily conjugate, and hence the lens spaces $L_\rho=S^{2m-11}/\langle\rho(AB^2)\rangle$ and $L_{\rho'}=S^{2m-11}/\langle\rho'(AB^2)\rangle$ are isometric. 
This completes the proof. 
\end{proof}

\subsection{An explicit infinite family}

In this subsection we show that the hypotheses of Theorem~\ref{thm:main} can be satisfied uniformly
for every prime $m\geq 13$.

\begin{corollary}\label{cor:explicit-family}
	Let $m\geq 13$ be a prime number, and set $h=\frac{m-1}{2}$, $p=h-2=\frac{m-5}{2}$, $U_m=\{1,\ldots,h\}$, $T=\{h-1,h\}$ and $T'=\{h-2,h\}$.	Write $U_m\smallsetminus T=\{k_1,\ldots,k_p\}$ and $U_m\smallsetminus T'=\{k'_1,\ldots,k'_p\}$, and set
	$\rho=\bigoplus_{j=1}^{p}\rho_{k_j}$ and $\rho'=\bigoplus_{j=1}^{p}\rho_{k'_j}$. Then, the spherical space forms $S^{2m-11}/\rho(\Gamma)$ and $S^{2m-11}/\rho'(\Gamma)$ are isospectral but not strongly isospectral.
\end{corollary}

\begin{proof}
	Using $2h\equiv-1\pmod m$, a direct verification shows that
	$T$ and $T'$ satisfy the congruence hypotheses of
	Theorem~\ref{thm:main}.
	
	It remains to prove that $L_\rho$ and $L_{\rho'}$ are not isometric.
	For a subset $E\subseteq\Z_m^\times$, we write $\pm E =E\cup(-E).$
	Note that the lens space $L_\rho$ (resp. $L_{\rho'}$) is isometric to the lens space whose multiset of parameters is $\pm(U_m\smallsetminus T)$ (resp. $\pm(U_m\smallsetminus T').$)
	
	Suppose that $L_\rho$ and $L_{\rho'}$ are isometric. Since these
	multisets are invariant under multiplication by $-1$, the lens space
	isometry criterion (see Subsection~\ref{subsec:lensspaces}) would give an integer $a$ relatively prime to $m$
	such that $a(\pm(U_m\smallsetminus T)) =	\pm(U_m\smallsetminus T').$
	Since $\pm U_m=\Z_m^\times$	and multiplication by $a$ permutes $\Z_m^\times$, we obtain $a(\pm T)=\pm T'.$
	
	Set $x=h-1, y=h, x'=h-2$ and $y'=h.$  Thus $\pm T=\{\pm x,\pm y\}$ and $\pm T'=\{\pm x',\pm y'\}.$
	Since multiplication by $a$ sends the pair $\{\pm x,\pm y\}$ onto $\{\pm x',\pm y'\}$, either
	\begin{equation}\label{eq:alternatives}
		yx'\equiv\pm y'x\pmod m
		\qquad\text{or}\qquad
		yy'\equiv\pm xx'\pmod m.
	\end{equation}
	
	The first alternative gives	$h(h-2)\equiv\pm h(h-1)\pmod m$, thus $h-2\equiv\pm(h-1)\pmod m$ because $h$ is invertible modulo $m$.
	If the sign is positive, then $m$ divides $1$. If the sign is
	negative, then $2h-3\equiv0\pmod m.$
	Together with $2h\equiv-1\pmod m$, this implies that $m$ divides $4$.
	
	The second alternative in \eqref{eq:alternatives} gives $h^2\equiv\pm(h-1)(h-2)\pmod m.$
	If the sign is positive, then $3h-2\equiv0\pmod m$ and
	multiplying this congruence by $2$ and using
	$2h\equiv-1\pmod m$ we obtain that $m$ divides $7$.
	If the sign is negative, then $2h^2-3h+2\equiv0\pmod m$, and
	multiplying this congruence by $4$ and using $2h\equiv-1\pmod m$ and $4h^2\equiv1\pmod m$ we obtain that $m$ divides $16$.

None of these possibilities can occur because $m$ is prime greater than $12$. 
Therefore, $L_\rho$ and $L_{\rho'}$ are not isometric, and the conclusion follows from Theorem~\ref{thm:main}.
\end{proof}

\begin{example}\label{ex:first-example}
	For $m=13$, as in Cororally~\ref{cor:explicit-family}, we have $h=6$, $p=4$, $U_{13}\smallsetminus T=\{1,2,3,4\}$, $U_{13}\smallsetminus T'=\{1,2,3,5\}$, and the corresponding representations of
	$\Gamma=\Gamma_2(13,4,12)$ are
	\begin{equation*}
		\rho
		=
		\rho_1\oplus\rho_2\oplus\rho_3\oplus\rho_4,
		\qquad
		\rho'
		=
		\rho_1\oplus\rho_2\oplus\rho_3\oplus\rho_5.
	\end{equation*}
	The associated lens spaces are
	\begin{equation*}
		L_\rho=L(13;1,1,2,2,3,3,4,4)
	    \qquad\text{and}\qquad
		L_{\rho'}=L(13;1,1,2,2,3,3,5,5).
	\end{equation*}
The appendix \cite{Lauret-appendix} to the article \cite{Lauret-computationalstudy} shows on page~181 that $L_{\rho}$ (denoted there by $L_{492}$) is isospectral to $L_{\rho'}$ (denoted there by $L_{495}$).
\end{example}

	\bibliographystyle{plain}

\end{document}